\documentclass[12pt]{article}
\usepackage{amssymb,amsmath,mathrsfs,amsthm,latexsym}
\usepackage{hyperref}
\usepackage{graphicx}
\usepackage{color}
\usepackage[OT2,OT1]{fontenc}
\usepackage{upquote}
\usepackage[numbers,sort&compress]{natbib}
\usepackage{alphabeta}
\usepackage{alltt} %inserting math to verbatim

\usepackage{xcolor}
\usepackage{framed}
\definecolor{shadecolor}{cmyk}{0,0,0,0.05} %gray

\newtheorem{theorem}{T{\hskip 0pt\footnotesize\bf HEOREM}}[section]
\newtheorem{lemma}[theorem]{L{\hskip 0pt\footnotesize\bf EMMA}}

\begin{document}

\title{\LARGE\bf Application of a new iterative formula for computing $\pi$ and nested \\
radicals with roots of $2$}

%\bigskip
\author{
\normalsize\bf Sanjar M. Abrarov, Rehan Siddiqui, Rajinder Kumar Jagpal \\
\normalsize\bf and Brendan M. Quine}

\date{November 7, 2025}
\maketitle
%\vspace{1cm}%\bigskip

\begin{abstract}
\abstract{In this work, we obtain an iterative formula that can be used for computing digits of $\pi$ and nested radicals of kind $c_n/\sqrt{2 - c_{n - 1}}$, where $c_0 = 0$ and $c_n = \sqrt{2 + c_{n - 1}}$. We also show how with the help of this iterative formula, the two-term Machin-like formulas for $\pi$ can be generated and approximated. Some examples with Mathematica codes are presented.}
\vspace{0.2cm}
\\
\end{abstract}

\section{Introduction}

Throughout many centuries, computing digits of $\pi$ remained a significant challenge \cite{Beckmann1971,Berggren2004,Borwein2008,Agarwal2013}. However, in 1876, English astronomer and mathematician John Machin found an efficient method to resolve this problem. Historically, he was the first to calculate over $100$ digits of $\pi$. In his approach, John Machin discovered and then used the following remarkable formula \cite{Beckmann1971,Berggren2004,Borwein2008,Agarwal2013}:
\begin{equation}
\label{MF4P} % Machin formula for pi
\frac{\pi}{4} = 4\arctan\left(\frac{1}{5}\right) - \arctan\left(\frac{1}{239}\right).
\end{equation}
The identities of kind
\begin{equation}\label{GMLF4P} % Generric Machin-like formula for pi
\frac{\pi}{4} = \sum_{j = 1}^J{A_j\arctan\left(\frac{1}{B_j}\right)}, \qquad A_j, B_j \in \mathbb{Q}
\end{equation}
are named after him as the Machin-like formulas for $\pi$.

The arctangent terms in the Machin-like formulas for $\pi$ can be computed by using the Maclaurin expansion series
\begin{equation}
\label{MSE} % Maclaurin series expansion
\arctan(x) = x - \frac{x^3}{3} + \frac{x^5}{5} - \frac{x^7}{7} + \cdots\, = \sum_{n = 0}^\infty{\frac{(-1)^n x^{2n + 1}}{(2n + 1)}}, \qquad x \le\left| 1 \right|.
\end{equation}
Since according to this equation, we get
\[
\arctan(x) = x + O(x^3),
\]
the convergence of the Machin-like formulas \eqref{GMLF4P} for $\pi$ is always better when the values $B_j$ are larger by an absolute value.

To estimate the efficiency of the Machin-like formulas for $\pi$, Lehmer introduced a measure, defined as follows \cite{Lehmer1938, Wetherfield1996, Gasull2023}:
\[
\mu = \sum_{j = 1}^J{\frac{1}{\log_{10}\left(\left|B_j\right|\right)}}.
\]
According to this formula, smaller measures indicate higher efficiency of the Machin-like formula for $\pi$. Lehmer's measure is smaller at larger absolute values of $B_j$ and at smaller number of the summation terms $J$.

Generally, we should imply that Lehmer's measure is valid only if all coefficients $B_j$ are integers. Otherwise, if $B_j \notin\mathbb{Z}$, its fractional parts
\[
\left\{B_j\right\} = B_j - \lfloor{B_j}\rfloor
\]
may cause further computational complexities requiring more usage of the computer memory and considerably extending the run-time in computing digits of $\pi$ \cite{Gasull2023}. As the fractional part is not desirable in the computation of the digits of $\pi$, it may be more preferable to apply the Machin-like formulas where all coefficients $B_j$ are integers.

The following Machin-like formula was discovered by Gauss \cite{Lehmer1938}:
\[
\frac{\pi}{4} = 12\arctan\left(\frac{1}{18}\right) + 8\arctan\left(\frac{1}{57}\right) - 5\arctan\left(\frac{1}{239}\right),
\]
As this formula has Lehmer's measure $\mu \approx 1.78661$ smaller than the original Machin-like Formula \eqref{MF4P} with measure $\mu \approx 1.85113$, it is more efficient for computing digits of $\pi$.

In 2002, more than one trillion decimal digits of $\pi$ were computed for the first time by a group of computer scientists lead by Yasumasa Kanada. To achieve this world record, the following Machin-like formulas were used~\cite{Calcut2009, Agarwal2013}:
\[
\begin{aligned}
\frac{\pi }{4} &= 44\arctan \left(\frac{1}{57}\right)+7\arctan\left(\frac{1}{239}\right) - 12\arctan\left(\frac{1}{682}\right) \\ 
&+24\arctan\left(\frac{1}{12943}\right)  
\end{aligned}
\]
and
\[
\begin{aligned}
\frac{\pi}{4} &= 12\arctan\left(\frac{1}{49}\right) + 32\arctan\left(\frac{1}{57}\right) - 5\arctan\left(\frac{1}{239}\right) \\ 
&+12\arctan\left(\frac{1}{110443}\right)  
\end{aligned}
\]
with Lehmer's measures $1.58604$ and $1.7799$, respectively, as a self-checking pair. These two formulas, known as the St\"{o}rmer--Takano pair, were named after Carl St\"{o}rmer and Kikuo Takano for their discoveries.

Although the current record, based on the Chudnovsky brothers' formula for $\pi$, exceeds one hundred trillion digits of $\pi$ \cite{Gasull2023}, application of the Machin-like formulas may be promising to calculate a comparable number of digits due to the availability of more advanced supercomputers than those used by Kanada almost $25$ years ago. Moreover, many other Machin-like formulas with smaller Lehmer's measures have been discovered \cite{Chien-Lih1997, Nimbran2010}. For example, the following two formulas (see \cite{Abrarov2024} for more details):
\begin{equation}\label{SChP1} % self-checking pair, equation 1
\begin{aligned}
\frac{\pi}{4} &= 83\arctan \left(\frac{1}{107}\right) + 17\arctan\left(\frac{1}{1710}\right) - 22\arctan\left(\frac{1}{103697}\right) \\
&-24\arctan\left(\frac{1}{2513489}\right) - 44\arctan\left(\frac{1}{18280007883}\right) \\ 
&+12\arctan\left(\frac{1}{7939642926390344818}\right) \\ 
&+22\arctan\left(\frac{1}{3054211727257704725384731479018}\right)
\end{aligned}
\end{equation}
and
\begin{equation}\label{SChP2} % self-checking pair, equation 2
\begin{aligned}
\frac{\pi}{4} &= 83\arctan\left(\frac{1}{107}\right) + 17\arctan\left(\frac{1}{1710}\right) - 22\arctan\left(\frac{1}{103697}\right) \\ 
&-12\arctan\left(\frac{1}{1256744}\right) - 22\arctan\left(\frac{1}{9140003941}\right) \\ 
&+12\arctan\left(\frac{1}{3158812219818}\right) \\ 
&+22\arctan\left(\frac{1}{167079344092131066905}\right),
\end{aligned}
\end{equation}
may be implemented as a self-checking pair more efficiently, as their Lehmer's measures are significantly smaller ($\mu \approx 1.34085$ for Equation \eqref{SChP1} and $\mu \approx 1.39524$ for Equation \eqref{SChP2}). Therefore, the Machin-like formulas with small Lehmer's measures have significant potential and can be competitive for computing digits of $\pi$.

In our previous publication \cite{Abrarov2017a}, we proposed a method for deriving the two-term Machin-like formula for $\pi$ in the following form:
\begin{equation}
\label{TTMLF4P} % two-term Machin-like formula for pi
\frac{\pi}{4} = 2^{k - 1}\arctan\left(\frac{1}{\alpha}\right) + \arctan\left(\frac{1}{\beta }\right),
\end{equation}
where $\alpha$ is an integer (see \cite{OEIS1999} for values of $\alpha$)
\[
\alpha  = \left\lfloor\frac{c_k}{\sqrt{2 - {c_{k - 1}}}}\right\rfloor
\]
that can be computed using nested radicals such that 
\[
c_0 = 0, \quad c_k = \sqrt{2 + {c_{k - 1}}}
\]
and 
\begin{equation}
\label{F4B} % formula for beta
\beta  = \frac{2}{\left[(\alpha + i)/(\alpha - i)\right]^{2^{k - 1}} - i} - i.
\end{equation}

As Equation \eqref{TTMLF4P} was constructed on the basis of the following identity \cite{Abrarov2017a},
$$
\frac{\pi}{4} = 2^{k - 1}\arctan\left(\frac{\sqrt{2 - c_{k - 1}}}{c_k}\right),
$$
its second arctangent function with argument $1/\beta$ can be interpreted as a small error term $\delta$ since
$$
2^{k - 1}\arctan\left(\frac{\sqrt{2 - c_{k - 1}}}{c_k}\right) = 2^{k - 1}\arctan\left(\frac{1}{\alpha}\right) + \delta,
$$
where
$$
\delta = \arctan\left(\frac{1}{\beta }\right).
$$
Therefore, the coefficient $\beta$ determines the magnitude of this error term $\delta$.

Recently, Gasull {et al.} proposed an alternative method to obtain the two-term Machin-like formula for $\pi$ of kind \eqref{TTMLF4P}. In their publication, they suggested an iterative method based on $R_j(n,x)$ functions that can be defined as follows~\cite{Gasull2023}:
\[
R_j(n,x) = \tan(n\sigma + j\pi/4), \qquad x = \tan(\sigma), \qquad n,j \in \mathbb{N}.
\]

Motivated by recent publications \cite{Campbell2023, Maritz2023, Spichal2024, Alferov2024, Kristiansen2025, Campbell2025, Cho2025, Spichal2025}, we propose a more efficient approach for the determination of the coefficient $\beta$ for Equation \eqref{TTMLF4P}. Specifically, application of Formula \eqref{F4B} for determination of the coefficient $\beta$ is not optimal. As integer $k$ increases the exponent $2^{k - 1}$ in the denominator increases very rapidly. As a consequence, this drastically delays computation. To resolve this problem, in our earlier work, we proposed a two-step iteration procedure \cite{Abrarov2017a} that will be discussed in the next section. However, the two-step iteration involves squaring that doubles the number of digits at each consecutive cycle of iteration. Although, compared to Equation \eqref{F4B}, the two-step iterative procedure is more efficient, it remains computationally costly at larger values of $k$. Furthermore, the two-step iteration cannot be used in approximation, as its application gives incorrect (or divergent) results.

To address these issues, we propose a new iterative formula that can be used for accurate approximation and more rapid determination of the required constant in the two-term Machin-like formula for $\pi$. Furthermore, we also show how this formula can be implemented to compute the constant $\pi$ and nested radicals consisting of square roots of $2$. Some numerical results with Mathematica codes are provided. To the best of our knowledge, this approach is new and has never been reported in scientific literature.

The outline of the remaining parts of this article is as follows. In Section \ref{sec2}, the previous methodology for generating the two-term Machin-like formula for $\pi$ using two-step iteration is presented. This section also shows how the quotients in the arctanent terms can be transformed into reciprocal integers. Section \ref{sec3} represents the main part of this work showing the derivation of the new iterative formula that can be used for generating and approximating the two-term Machin-like formula for $\pi$. Finally, Section \ref{sec4} describes how the new iterative formula can be implemented.

%%%%%%%%%%%%%%%%%%%%%%%%%%%%%%%%%%%%%%%%%%
\section{Preliminaries}\label{sec2}

As has been mentioned above, Equation \eqref{F4B} should be avoided for computation of the constant $\beta$ at $k \gg 1$. The following theorem shows how else the constant $\beta$ can be calculated.

\begin{theorem}
\label{theorem2.1}
The following formula
\begin{equation}
\label{F4B2} % formula for betta 2
\beta  = \frac{\kappa_k}{1 - \lambda_k},
\end{equation}
where
\begin{equation}
\label{TSI} % two-step iteration
\left\{
\begin{aligned}
\kappa_n &= \kappa_{n - 1}^2 - \lambda_{n - 1}^2 \\
\lambda_n &= 2\kappa_{n - 1}\lambda_{n - 1}
\end{aligned}
\right. \qquad n = \{2,3,4, \ldots k\},
\end{equation}
with initial values
\[
\kappa_1 = \frac{\alpha^2 - 1}{\alpha^2 + 1}
\]
and
\[
\lambda_1 = \frac{2\alpha}{\alpha^2 + 1}
\]
holds.
\end{theorem}

\begin{proof}
By induction, it follows that
\[
\begin{aligned}
\left(\frac{\alpha + i}{\alpha - i}\right)^{2^0} &= \frac{2\alpha}{\alpha^2 + 1} + i\frac{\alpha^2 - 1}{\alpha^2 + 1} = \kappa_1 + i\lambda_1, \\
\left(\frac{\alpha + i}{\alpha - i}\right)^{2^1} &= (\kappa_1 + i\lambda_1)^2 = \kappa_1^2 - \lambda_1^2 + 2i\kappa_1\lambda_1 = \kappa_2 + i\lambda_2, \\
\left(\frac{\alpha + i}{\alpha - i}\right)^{2^2} &= (\kappa_2 + i\lambda_2)^2 = \kappa_2^2 - \lambda_2^2 + 2i\kappa_2\lambda_2 = \kappa_3 + i\lambda_3, \\
&\vdots \\
\left(\frac{\alpha + i}{\alpha - i}\right)^{2^n} &= (\kappa_n + i\lambda_n)^2 = \kappa_n^2 - \lambda_n^2 + 2i\kappa_n\lambda_n = \kappa_{n + 1} + i\lambda_{n + 1},
\end{aligned}
\]
Consequently, we get a useful identity: 
\begin{equation}
\label{PRF} % power reduction formula
\begin{aligned}
\left(\frac{\alpha + i}{\alpha - i}\right)^{2^{k - 1}} = \kappa_k + i\lambda_k.
\end{aligned}
\end{equation}

Substituting Identity \eqref{PRF} into Equation \eqref{F4B} yields the following:
\begin{equation}
\label{F4B3} % formula for betta 3
\beta  = \frac{2}{\kappa_k + i\lambda_k - i} - i = \frac{2\kappa_k}{\kappa_k^2 + (\lambda_k - 1)^2} + i\left(\frac{2(1 - \lambda_k)}{\kappa_k^2 + (\lambda_k - 1)^2} - 1 \right).
\end{equation}
Since $\beta$ is a real number, the imaginary part of the equation above must be equal to zero. Consequently, we have 
\[
\frac{2(1 - \lambda_k)}{\kappa_k^2 + (\lambda_k - 1)^2} - 1 = 0
\]
or
\[
2(1 - \lambda_k) = \kappa_k^2 + (\lambda_k - 1)^2
\]
or
\begin{equation}
\label{F4K} % formula forr kappa
\kappa_k^2 = 2(1 - \lambda_k) - (\lambda_k - 1)^2.
\end{equation}
Substituting Equation \eqref{F4K} into the real part of Equation \eqref{F4B3}, we obtain Equation \eqref{F4B2}. This completes the proof.
\end{proof}

In our earlier publication \cite{Abrarov2022}, we show how to efficiently generate the multi-term Machin-like formulas for $\pi$ using the following equation-template:
\begin{equation}
\label{ET} % equation-template
\frac{\pi}{4} = 2^{k - 1}\arctan\left(\frac{1}{\gamma_k}\right) + \left(\sum_{m = 1}^M{\arctan\left(\frac{1}{\left\lfloor\theta_{m,k}\right\rfloor}\right)}\right) + \arctan\left(\frac{1}{\theta_{M + 1,k}}\right), \\
\end{equation}
where
\[
\theta_{m,k} = \frac{1 + \left\lfloor\theta_{m - 1,k}\right\rfloor\theta_{m - 1,k}}{\left\lfloor\theta_{m - 1,k}\right\rfloor - \theta_{m - 1,k}}, \qquad m \ge 2.
\]

In Formula \eqref{ET}, we imply that $M \ge 0$ such that at $M = 0$, the sum
\[
\sum_{m = 1}^M{\arctan\left(\frac{1}{\left\lfloor\theta_{m,k}\right\rfloor}\right)}
\]
is equal to zero. Consequently, when $M = 0$, Equation \eqref{ET} is reduced to
\begin{equation}
\label{R2TTMLF4P} % reduced to two-term Machin-like formula for pi
\frac{\pi}{4} = 2^{k - 1}\arctan\left(\frac{1}{\gamma_k}\right) + \arctan\left(\frac{1}{\theta_{1,k}} \right),
\end{equation}
where ${\gamma _k} = \alpha$ and $\theta_{1,k} = \beta$ according to Equation \eqref{TTMLF4P} above. In fact, Equation \eqref{ET} can be obtained from Equation \eqref{TTMLF4P} together with the following identity \cite{Abrarov2022}:
\[
\arctan \left( {\frac{1}{z}} \right) = \arctan \left( {\frac{1}{{\left\lfloor z \right\rfloor }}} \right) + \arctan \left( {\frac{{\left\lfloor z \right\rfloor  - z}}{{1 + z\left\lfloor z \right\rfloor }}} \right), \qquad z \notin [0,1).
\]

Consider how the two well-known Machin-like formulas for $\pi$ can be derived using Equation \eqref{ET}. At $k = 2$ and $M = 0$, Equation \eqref{ET} leads to
\begin{equation}
\label{HF4P} % Hermann's formula for pi
\frac{\pi}{4} = 2\arctan\left(\frac{1}{2}\right) - \arctan\left(\frac{1}{7}\right).
\end{equation}
This equation is commonly known as Hermann's formula for $\pi$ \cite{Stormer1899}. At $k = 3$ and $M = 0$, Equation \eqref{ET} leads to the original Machin Formula \eqref{MF4P} for $\pi$.

Application of Equation \eqref{ET} may also be a useful technique to transform the arctangent term with the quotient $\theta_{1,k} = \beta$ into the sum of arctangents with reciprocal integers. For example, at $k = 4$ we get
\[
\gamma_4 = \left\lfloor\frac{c_4}{\sqrt{2 - c_3}}\right\rfloor = \left\lfloor\frac{\sqrt{2 + \sqrt{2 + \sqrt{2 + \sqrt 2}}}}{\sqrt{2 - \sqrt{2 + \sqrt{2 + \sqrt 2}}}}\right\rfloor = 10.
\]
Consequently, applying two-step iteration \eqref{TSI}, we can find that
\[
\theta_{1,4} = -\frac{147153121}{1758719}.
\]
Substituting these values into Equation \eqref{ET} leads to
\begin{equation}
\label{ITTMLF4P} % initial two-term Machin-like formula for pi
\frac{\pi}{4} = 8\arctan\left(\frac{1}{10}\right) - \arctan\left(\frac{1758719}{147153121}\right).
\end{equation}

As has been mentioned above, it is desirable to apply reciprocal integers rather than quotients. Equation \eqref{ET} can be used to transform the quotients into reciprocal integers. For example, at $k = 4$ and $M = 1$ from Equation \eqref{ET}, it follows that
\[
\frac{\pi}{4} = 8\arctan\left(\frac{1}{10}\right) - \arctan\left(\frac{1}{84}\right) - \arctan\left(\frac{579275}{12362620883}\right).
\]
At $k = 4$ and $M = 2$, Equation \eqref{ET} gives the following:
\[
\begin{aligned}
\frac{\pi}{4} &= 8\arctan\left(\frac{1}{10}\right) - \arctan\left(\frac{1}{84}\right) - \arctan\left(\frac{1}{21342}\right) \\
&-\arctan\left(\frac{266167}{263843055464261}\right).
\end{aligned}
\]

Repeatedly incrementing the integer $M$ at $k = 4$ and $M = 5$, Equation \eqref{ET} finally yields the $7$-term Machin-like formula:
\begin{equation}
\label{STMLF4P} % 7-terms Machin-like formula for pi 
\begin{aligned}
\frac{\pi}{4} &= 8\arctan\left(\frac{1}{10}\right) - \arctan\left(\frac{1}{84}\right) - \arctan\left(\frac{1}{21342}\right) \\
&-\arctan\left(\frac{1}{991268848}\right) - \arctan\left(\frac{1}{193018008592515208050}\right) \\
&-\arctan\left(\frac{1}{197967899896401851763240424238758988350338}\right) \\
&-\arctan\left(\frac{1}{\Omega}\right),
\end{aligned}
\end{equation}
\[
\begin{aligned}
\Omega  = &\,11757386816817535293027775284419412676799191500853701 \ldots \\
&\,\,8836932014293678271636885792397,
\end{aligned}
\]
where all arctangent arguments are reciprocal integers.

As we can see, Hermann's \eqref{HF4P}, Machin's \eqref{MF4P}, and the generated \eqref{STMLF4P} formulas belong to the same generic group, as all of them can be derived using the same equation-template~\eqref{ET}.

The Machin-like formulas \eqref{GMLF4P} for $\pi $ can be validated using the following relation:
\begin{equation}
\label{VF} % validation formula
\prod_{j = 1}^J{\left(B_j + i\right)^{A_j}} \propto (1 + i).
\end{equation}
The right side of this relation implies that the real part of the product must be equal to its imaginary part, as follows:
\[
\Re\left\{\prod_{j = 1}^J{\left(B_j + i\right)^{A_j}}\right\} = \Im\left\{\prod_{j = 1}^J{\left(B_j + i\right)^{A_j}}\right\}.
\]
For example, the original Machin Formula \eqref{MF4P} for $\pi$ can be readily validated by applying the following relation:
\[
(5 + i)^4 (259 + i)^{-1} = 2(1 + i)
\]
since the real and imaginary parts of the product are equal to the same number, $2$. 

The following Mathematica code validates Equation \eqref{STMLF4P}:
\vspace{-0.25cm}
\small
\begin{shaded}
\begin{verbatim}
(* Define long string *)
longStr = StringJoin["11757386816817535293027775284419412676",
    "7991915008537018836932014293678271636885792397"];

(* Computing the coefficient *)
coeff = (10 + I)^8*(84 + I)^-1*(21342 + I)^-1*(991268848 + I)^-1*
    (193018008592515208050+I)^-1*
		    (197967899896401851763240424238758988350338 + I)^-1*
            (FromDigits[longStr]+I)^-1;

Print[Re[coeff] == Im[coeff]];
\end{verbatim}
\end{shaded}
\normalsize
\vspace{-0.25cm}
\noindent by returning the output {\bf\ttfamily True}.

Once all quotients are transformed into reciprocal integers, we can compute the digits of $\pi$ by using the Maclaurin series expansion \eqref{MSE} of the arctangent function. However, our empirical results show that the following two expansion series,
\begin{equation}
\label{ESE} % Euler's series expansion
\arctan(x) = \sum_{n = 0}^\infty{\frac{{2^{2n}}{(n!)^2}}{(2n + 1)!}\frac{x^{2n + 1}}{(1 + x^2)^{n + 1}}}
\end{equation}
and
\begin{equation}
\label{ASE} % alternative series expansion
\arctan(x) = 2\sum_{n = 1}^\infty{\frac{1}{2n - 1}\frac{g_n(x)}{g_n^2(x) + h_n^2(x)}},
\end{equation}
where
\[
g_1(x) = 2/x, \qquad h_1(x) = 1,
\]
\[
g_n(x) = g_{n - 1}(x)(1 - 4/x^2) + 4h_{n - 1}(x)/x,
\]
\[
h_n(x) = h_{n - 1}(x)(1 - 4/x^2) - 4g_{n - 1}(x)/x,
\]
can be used more efficiently for computing digits of $\pi$ due to more rapid convergence.

Equation \eqref{ESE} is known as Euler's series expansion. It is interesting to note that this series expansion can be derived from the following integral: \cite{Chien-Lih2005}
\[
\arctan(x) = \int_0^{\pi/2}{\frac{x\sin u}{1 + {x^2}}\frac{1}{1 - \frac{x^2\sin^2 u}{1 + x^2}}du}.
\]

Equation \eqref{ASE} can be derived by substituting
\[
f(x,t) = \frac{x}{2}\left(\frac{1}{1 + ixt} + \frac{1}{1 - ixt}\right)
\]
into the identity
\[
\int_0^1 f(x,t)dt = {\left.2\sum_{m = 1}^M{\sum_{n = 0}^\infty{\frac{1}{(2M)^{2n + 1}(2n + 1)!}\frac{\partial ^{2n}}{\partial t^{2n}}f(x,t)}}\right|}_{t = \frac{m - 1/2}{M}}
\]
that we proposed and used earlier (see \cite{Abrarov2023} and the literature therein for more details). A computational test we performed shows that the arctangent series expansion \eqref{ASE} is significantly faster in convergence than the arctangent series expansion \eqref{ESE}.

At $k = 6$, with the help of two-step iteration \eqref{TSI}, we {obtain} the following:
\[
\theta_{1,6} = -\frac{2634699316100146880926635665506082395762836079845121}{38035138859000075702655846657186322249216830232319}.
\]
Consequently, the two-term Machin-like formula for $\pi$ {becomes} as follows:
\[
\begin{aligned}
\frac{\pi}{4} &= 32\arctan\left(\frac{1}{40}\right) \\
&-\arctan\left(\frac{38035138859000075702655846657186322249216830232319}{2634699316100146880926635665506082395762836079845121}\right).
\end{aligned}
\]

As we can see, this equation contains a quotient with a large number of digits in the numerator and denominator. We may attempt to reduce the number of digits by approximation. Unfortunately, the two-step iteration \eqref{TSI} is not efficient in approximating $\theta_{1,k}$. Any attempt of ours to approximate $\theta_{1,k}$ via two-step iteration \eqref{TSI} either do not provide the desired accuracy or completely diverge from the value $\theta_{1,k}$. This makes approximation inefficient and unpredictable, especially at larger values of the integer $k$.

Moreover, the number of digits rapidly grows with increasing $k$. For example, at $k = 27$ and $M = 0$, Equation \eqref{ET} results in
\begin{equation}
\label{K27}
\begin{aligned}
\frac{\pi}{4} &= 2^{27 - 1}\arctan\left(\frac{1}{\gamma _{27}}\right) + \arctan\left(\frac{1}{\theta_{1,27}}\right) \\
&= 67108864\arctan\left(\frac{1}{85445659}\right) - \arctan\left(\frac{9732933578 \ldots 4975692799}{2368557598 \ldots 9903554561}\right),
\end{aligned}
\end{equation}
where the argument in the second arctangent function contains $522\text{,}185\text{,}807$ digits in the numerator and $522\text{,}185\text{,}816$ digits in the denominator. In a recent publication, Gasull {et al.} showed that at $k = 31$, the number of digits in the numerator and denominator in the two-term Machin-like formula for $\pi$ of kind \eqref{TTMLF4P} are $9\text{,}647\text{,}887\text{,}023$ and $9\text{,}647\text{,}887\text{,}033$, respectively (see Table 2 in \cite{Gasull2023}). Consequently, two-step iteration \eqref{TSI} appears to be impractical for approximating the constants ${\theta _{1,k}}$.

This problem can be effectively resolved by using a new method of iteration that will be shown in the next section.

%%%%%%%%%%%%%%%%%%%%%%%%%%%%%%%%%%%%%%%%%%

\section{An Iterative Formula to Compute \boldmath{$\pi$}}\label{sec3}

A problem that appears with two-step iteration \eqref{TSI} is a rapidly growing number of digits. In particular, the number of digits in $\kappa_n$ and ${\lambda _n}$ doubles at each consecutive increment of the index $n$. This may also restrict the application of the two-step \mbox{iteration \eqref{TSI}}. The following two theorems shows how a new iteration technique can be derived as a more efficient alternative to the two-step iteration \eqref{TSI}.

\begin{theorem}
\label{theorem3.1}
The following formula
\[
\frac{\pi}{4} = \arctan\left(\frac{1}{u}\right) + \arctan\left(\frac{u - 1}{u + 1}\right)
\]
holds.
\end{theorem}

\begin{proof}
The proof immediately follows from the identity
\begin{equation}
\label{AI} % arctangent identity
\arctan(x) + \arctan(y) = \arctan\left(\frac{x + y}{1 - xy}\right).
\end{equation}
Specifically, assuming
\[
x = \frac{1}{u}
\]
and 
\[
y = \frac{u - 1}{u + 1},
\]
according to identity \eqref{AI}, we have
\[
\arctan\left(\frac{1/u + (u - 1)/(u + 1)}{1 - 1/u\left((u - 1)/(u + 1)\right)}\right) = \arctan(1) = \frac{\pi}{4}.
\]
This completes the proof.
\end{proof}

We now consider the next theorem.

\begin{theorem}
\label{theorem3.2}
The following relation
\[
2^{k - 1}\arctan\left(\frac{1}{\gamma_k}\right) = \arctan\left(\frac{1}{v_k}\right),
\]
where
\[
v_n = \frac{1}{2}\left(v_{n - 1} - \frac{1}{v_{n - 1}}\right), \qquad n = \{2,3, \ldots k\}
\]
and $v_1 = \gamma_k$ holds.
\end{theorem}

\begin{proof}
Since
$$
{2^0}\arctan \left( {\frac{1}{{{\gamma _1}}}} \right) = \arctan \left( {\frac{1}{{{v_1}}}} \right),
$$
it follows that
$$
{2^1}\arctan \left( {\frac{1}{{{\gamma _2}}}} \right) = \arctan \left( {\frac{1}{{{v_2}}}} \right),
$$
where
\[
\frac{1}{{{v_2}}} = \frac{{2/{v_1}}}{{1 - 1/v_1^2}}.
\]
Similarly, at $k = 3$ we get
\[
2^2\arctan\left(\frac{1}{\gamma_3}\right) = \arctan\left(\frac{1}{v_3}\right),
\]
where
\[
\frac{1}{v_3} = \frac{2/v_2}{1 - 1/v_2^2}.
\]
Therefore, through induction for an arbitrary integer $k$, we can write
$$
2^{k - 1}\arctan\left(\frac{1}{\gamma_k}\right) = \arctan\left(\frac{1}{v_k}\right),
$$
where
\[
\frac{1}{v_k} = \frac{2/v_{k - 1}}{1 - 1/v_{k - 1}^2}.
\]

The equation above can be rewritten as
\begin{equation}
\label{NIF} % new iterative formula
v_k = \frac{1}{2}\left(v_{k - 1} - \frac{1}{v_{k - 1}}\right),
\end{equation}
and this proves the theorem.
\end{proof}

Thus, from Theorems \ref{theorem3.1} and \ref{theorem3.2} and Equation \eqref{R2TTMLF4P} it follows that
\begin{equation}
\label{MTTMLF4P} % modified two-term Machin-like formula for pi
\frac{\pi}{4} = 2^{k - 1}\arctan\left(\frac{1}{v_1}\right) + \arctan\left(\frac{{v_k - 1}}{v_k + 1} \right).
\end{equation}
where $v_k$ is related to $\theta_{1,k}$ as
\[
v_k = \frac{\theta_{1,k} + 1}{\theta_{1,k} - 1}.
\]

The iterative Formula \eqref{NIF} and Equation \eqref{MTTMLF4P} are the main results. The significance of these equations can be demonstrated by computing digits of $\pi$ and nested radicals consisting of square roots of $2$.

\section{An Iterative Formula to Compute \boldmath{$\pi$}}\label{sec3}

A problem that appears with two-step iteration \eqref{TSI} is a rapidly growing number of digits. In particular, the number of digits in $\kappa_n$ and ${\lambda _n}$ doubles at each consecutive increment of the index $n$. This may also restrict the application of the two-step \mbox{iteration \eqref{TSI}}. The following two theorems shows how a new iteration technique can be derived as a more efficient alternative to the two-step iteration \eqref{TSI}.

\begin{theorem}
\label{theorem3.1}
The following formula
\[
\frac{\pi}{4} = \arctan\left(\frac{1}{u}\right) + \arctan\left(\frac{u - 1}{u + 1}\right)
\]
holds.
\end{theorem}

\begin{proof}
The proof immediately follows from the identity
\begin{equation}
\label{AI} % arctangent identity
\arctan(x) + \arctan(y) = \arctan\left(\frac{x + y}{1 - xy}\right).
\end{equation}
Specifically, assuming
\[
x = \frac{1}{u}
\]
and 
\[
y = \frac{u - 1}{u + 1},
\]
according to identity \eqref{AI}, we have
\[
\arctan\left(\frac{1/u + (u - 1)/(u + 1)}{1 - 1/u\left((u - 1)/(u + 1)\right)}\right) = \arctan(1) = \frac{\pi}{4}.
\]
This completes the proof.
\end{proof}

We now consider the next theorem.

\begin{theorem}
\label{theorem3.2}
The following relation
\[
2^{k - 1}\arctan\left(\frac{1}{\gamma_k}\right) = \arctan\left(\frac{1}{v_k}\right),
\]
where
\[
v_n = \frac{1}{2}\left(v_{n - 1} - \frac{1}{v_{n - 1}}\right), \qquad n = \{2,3, \ldots k\}
\]
and $v_1 = \gamma_k$ holds.
\end{theorem}

\begin{proof}
Since
$$
{2^0}\arctan \left( {\frac{1}{{{\gamma _1}}}} \right) = \arctan \left( {\frac{1}{{{v_1}}}} \right),
$$
it follows that
$$
{2^1}\arctan \left( {\frac{1}{{{\gamma _2}}}} \right) = \arctan \left( {\frac{1}{{{v_2}}}} \right),
$$
where
\[
\frac{1}{{{v_2}}} = \frac{{2/{v_1}}}{{1 - 1/v_1^2}}.
\]
Similarly, at $k = 3$ we get
\[
2^2\arctan\left(\frac{1}{\gamma_3}\right) = \arctan\left(\frac{1}{v_3}\right),
\]
where
\[
\frac{1}{v_3} = \frac{2/v_2}{1 - 1/v_2^2}.
\]
Therefore, through induction for an arbitrary integer $k$, we can write
$$
2^{k - 1}\arctan\left(\frac{1}{\gamma_k}\right) = \arctan\left(\frac{1}{v_k}\right),
$$
where
\[
\frac{1}{v_k} = \frac{2/v_{k - 1}}{1 - 1/v_{k - 1}^2}.
\]

The equation above can be rewritten as
\begin{equation}
\label{NIF} % new iterative formula
v_k = \frac{1}{2}\left(v_{k - 1} - \frac{1}{v_{k - 1}}\right),
\end{equation}
and this proves the theorem.
\end{proof}

Thus, from Theorems \ref{theorem3.1} and \ref{theorem3.2} and Equation \eqref{R2TTMLF4P} it follows that
\begin{equation}
\label{MTTMLF4P} % modified two-term Machin-like formula for pi
\frac{\pi}{4} = 2^{k - 1}\arctan\left(\frac{1}{v_1}\right) + \arctan\left(\frac{{v_k - 1}}{v_k + 1} \right).
\end{equation}
where $v_k$ is related to $\theta_{1,k}$ as
\[
v_k = \frac{\theta_{1,k} + 1}{\theta_{1,k} - 1}.
\]

The iterative Formula \eqref{NIF} and Equation \eqref{MTTMLF4P} are the main results. The significance of these equations can be demonstrated by computing digits of $\pi$ and nested radicals consisting of square roots of $2$.

\section{Approximation Methodologies}\label{sec4}

\subsection{A Rational Approximation of $\pi$}

 Recently, we have reported a rational approximation for $\pi$ by approximating the two-term Machin-like Formula \eqref{R2TTMLF4P} \cite{Abrarov2025}. Here, we show how a rational approximation for $\pi$ can be constructed by using Equation \eqref{MTTMLF4P} based on iterative Formula \eqref{NIF}.

Consider the following theorem.

\begin{theorem}
\label{RNR} % Ramanujan's nested radical
The following limit
\[
\lim_{k\to\infty}c_k = 2
\]
holds.
\end{theorem}

\begin{proof}
This is Ramanujan's nested radical. Since
\[
c_k = \underbrace{\sqrt{2 + \sqrt{2 + \sqrt{2 + \ldots + \sqrt 2}}}}_{k\,\,{\rm square\,\,roots}} = \sqrt{2 + c_{k - 1}},
\]
we can write
\[
\lim_{k\to\infty}c_k = \lim_{k\to\infty}\sqrt{2 + {c_{k - 1}}} = \lim_{k\to\infty} \sqrt{2 + c_k}.
\]

Let
\[
z = \lim_{k\to\infty}c_k.
\]
Then, solving the equation
\[
z = \sqrt{2 + z}
\]
or
\[
z^2 - z - 2 = 0,
\]
we get two solutions, $z_1 =  - 1$ and $z_2 = 2$. Since the sequence
\[
\{c_0,c_1,c_2,c_3,\ldots\} = \left\{0,\sqrt{2},\sqrt{2+\sqrt{2}},\sqrt{2+\sqrt{2+\sqrt{2}}},\ldots\right\}
\]
monotonically increases and is positive at any $c_n$ except $c_0 = 0$, we have to exclude the negative solution. Consequently, the theorem is proved.
\end{proof}

Based on Theorem \ref{RNR}, we can conclude that
\begin{equation}
\label{L4CK} % limit for c_k
\lim_{k\to\infty}\sqrt{2 - c_{k}} = 0
\end{equation}
since due to relation
\[
\lim_{k\to\infty}c_k = 2
\]
leading to
\[
\sqrt{2 - \lim_{k\to\infty} c_k} = 0
\]
the limit \eqref{L4CK} follows.

Next, will we try to find a rational approximation for computing digits of $\pi$. This can be done using Equation \eqref{MTTMLF4P}. In order to approximate it, we need to consider the two lemmas below.

\begin{lemma}
The following limit
\label{L4AF} % lemma for arctangent function
\[
\frac{\pi}{4} = \lim_{k\to\infty}{2^{k - 1}\arctan\left(\frac{1}{\gamma _k}\right)}
\]
holds.
\end{lemma}

\begin{proof}
By definition,
\[
\frac{c_k}{\sqrt{2 - c_{k - 1}}} - 1 < \gamma_k < \frac{c_k}{\sqrt {2 - c_{k - 1}}}.
\]
Since 
\[
\lim_{k\to\infty}c_k = 2
\]
while
\[
\lim_{k\to\infty} \sqrt{2 - c_{k - 1}} = 0
\]
we can infer that
\[
\lim_{k\to\infty}\frac{c_k}{\sqrt{2 - c_{k - 1}}} = \infty.
\]
Therefore, we can write
\[
\lim_{k\to\infty}\frac{\gamma_k}{\sqrt{2 - c_{k - 1}}/c_k} = 1
\]
from which it follows that
\[
\lim_{k\to\infty}\arctan\left(\frac{\sqrt{2 - c_{k - 1}}}{c_k}\right) = \lim_{k\to\infty}\arctan\left(\frac{1}{\gamma_k}\right)
\]
and proof of this lemma follows since \cite{Abrarov2018}
\begin{equation}
\label{F4PWNR} % formula fo pi with nested radicals
\lim_{k\to\infty}2^{k - 1}\arctan\left(\frac{\sqrt{2 - c_{k - 1}}}{c_k}\right) = \frac{\pi}{4}.
\end{equation}
\end{proof}

\begin{lemma}
\label{L4RA} % limit for rational approximation
The following limit
\[
\frac{\pi}{4} = \lim_{k\to\infty} \left(\frac{2^{k - 1}}{\gamma_k} + \frac{1}{\theta_{1,k}}\right) = \lim_{k\to\infty}\left(\frac{2^{k - 1}}{v_1} + \frac{v_k - 1}{v_k + 1}\right)
\]
holds.
\end{lemma}

\begin{proof}
The proof follows immediately from Lemma \ref{L4AF} and due to Equation \eqref{MTTMLF4P}.
\end{proof}

The importance of Lemma \ref{L4RA} can be seen by comparing the accuracy of two approximations,	
\begin{equation}
\label{STRA} % single term rational approximation
\frac{\pi}{4} \approx \frac{2^{k - 1}}{\gamma_k} = \frac{2^{k - 1}}{v_1}
\end{equation}
and	
\begin{equation}
\label{DTRA} % double term rational approximation
\frac{\pi }{4} \approx \frac{2^{k - 1}}{\gamma_k} + \frac{1}{\theta_{1,k}} = \frac{2^{k - 1}}{v_1} + \frac{v_k - 1}{v_k + 1}.
\end{equation}

Since $v_k$ is very close to unity, the double-term rational approximation \eqref{DTRA} can be simplified as follows:
\begin{equation}
\label{SDTRA} % simplified double term rational approximation
\frac{\pi}{4} \approx \frac{2^{k - 1}}{v_1} + \frac{v_k - 1}{2}.
\end{equation}

Mathematica code below shows how a number of digits of $\pi$ can be computed using single- and double-term rational approximations, \eqref{STRA} and \eqref{SDTRA}, respectively.

\vspace{-0.25 cm}
\small
\begin{shaded}
\begin{verbatim}
Clear[k, c, flNum, accNum, v1, vk];

(* Integer k *)
k = 3000;

(* Use if needed: $RecursionLimit = 100000; *)
$RecursionLimit = 50000;

c[0] := c[0] = 0;
c[n_] := c[n] = SetAccuracy[Sqrt[2 + c[n - 1]], k];

(* Floor number *)
flNum = Floor[c[k]/Sqrt[2 - c[k - 1]]];

(* Accuracy number *)
accNum = Length[RealDigits[flNum][[1]]];

(* Coefficient v_1 *)
v1 = SetAccuracy[flNum, 2*accNum];

Print["At k = ", k, " number of digits of \[Pi] with single term: ", 
    MantissaExponent[\[Pi] - 4 (2^(k - 1)/v1)][[2]] // Abs];

(* Compute coefficient v_k *)
vk = v1;
Do[vk = 1/2*(vk - 1/vk), k - 1];

Print["At k = ", k, " number of digits of \[Pi] with two terms: ", 
    MantissaExponent[\[Pi] - 4 (2^(k - 1)/v1 
		    + (vk - 1)/2)][[2]] //Abs];
\end{verbatim}
\end{shaded}
\normalsize
\vspace{-0.25 cm}

This code generates the following output:
\small
\begin{verbatim}
At k = 3000 number of digits of π with single term: 902
At k = 3000 number of digits of π with two terms: 1805
\end{verbatim}
\normalsize

As we can see from this example, once we know the value of the constant ${v_k}$, at $k = 3000$, the number of digits of $\pi$ is doubled from $902$ to $1805$.

\subsection{An Approximation of $\pi$ with Cubic Convergence}

Now, we show how to obtain a formula for $\pi$ with cubic convergence. Consider the following theorem.

\begin{theorem}
The two-term Machin-like Formula \eqref{R2TTMLF4P} for $\pi$ can be represented in trigonometric form as:
\[
\frac{\pi}{4} = 2^{k - 1}\arctan\left(\frac{1}{\gamma_k}\right) + \arctan\left(\frac{1 - \sin\left(2^{k - 1}\arctan\left(\frac{2\gamma _k}{\gamma_k^2 - 1}\right)\right)}{\cos\left(2^{k - 1}\arctan\left(\frac{2\gamma_k}{\gamma_k^2 - 1}\right)\right)}\right).
\]
\end{theorem}

\begin{proof}
Applying de Moivre's theorem, we can write the following:
\[
\begin{aligned}
(\kappa_1 + i\lambda_1)^{2^{k - 1}} &= \left(\kappa_1^2 + i\lambda_1^2\right)^{2^{k - 2}}\left(\cos \left(2^{k - 1}{\rm Arg}(\kappa_1 + i\lambda_1)\right)\right) \\
&+i\sin\left(2^{k - 1}{\rm Arg}(\kappa_1 + i\lambda_1)\right)
\end{aligned}
\]

Substituting the equation above into Equation \eqref{F4B3}, and taking into consideration that $\theta_{1,k} = \beta$, we get
\[
\theta_{1,k} = \frac{\cos\left(2^{k - 1}{\rm Arg}(\kappa_1 + i\lambda_1)\right)}{1 - \sin\left(2^{k - 1}{\rm Arg}(\kappa_1 + i\lambda_1)\right)}.
\]

Since
\[
{\rm Arg}(x + iy) = \arctan\left(\frac{y}{x}\right), \qquad x > 0,
\]
it follows that
\[
{\rm Arg}(\kappa_1 + i\lambda_1) = \arctan\left(\frac{2\gamma_k}{\gamma_k^2 - 1}\right), \qquad x > 0.
\]
Consequently, Equation \eqref{F4B3} can be represented as
\begin{equation}
\label{TF4T} % trigonometric formula for theta
\theta_{1,k} = \frac{\cos\left(2^{k - 1}\arctan\left(\frac{2\gamma _k}{\gamma _k^2 - 1}\right)\right)}{1 - \sin\left(2^{k - 1}\arctan\left(\frac{2\gamma_k}{\gamma_k^2 - 1}\right)\right)}
\end{equation}
and this completes the proof.	
\end{proof}

We may attempt to approximate the coefficient $\theta_{1,k}$. Since
\[
\frac{2\gamma_k}{\gamma_k^2 - 1}\to\frac{2}{\gamma_k}
\]
with increasing $k$, Equation \eqref{TF4T} can be approximated as follows:
\[
\theta_{1,k} \approx \frac{\cos\left(2^{k - 1}\arctan\left(\frac{2}{\gamma_k}\right)\right)   
}{1 - \sin\left(2^{k - 1}\arctan\left(\frac{2}{\gamma_k}\right)\right)}.
\]

Since
\[
\arctan(x) \approx x, \qquad |x| \ll 1,
\]
the approximation above can be further simplified as follows:
\begin{equation}
\label{SF4T} % simplified formula for theta
\theta_{1,k} \approx \frac{\cos\left(\frac{2^k}{\gamma_k}\right)}{1 - \sin\left(\frac{2^k}{\gamma_k} \right)}.
\end{equation}

Using the following trigonometric identities,
\[
\cos(x) = \frac{2\tan(x/2)}{1 + \tan^2(x/2)}
\]
and
\[
\sin(x) = \frac{1 - \tan^2(x/2)}{1 + \tan^2(x/2)}
\]
Equation \eqref{SF4T} can be expressed as follows:
\[
\theta_{1,k} \approx \frac{1 + \tan^2\left(2^{k - 1}/\gamma_k\right)}{1 - \tan\left(2^{k - 1}/\gamma_k \right)}.
\]
This results in
\begin{equation}
\label{A4RT} % approximation for reciprocal of theta
\frac{1}{\theta_{1,k}} \approx \frac{1 - \tan \left(2^{k - 1}/\gamma_k\right)}{1 + \tan^2\left(2^{k - 1}/\gamma_k\right)}.
\end{equation}

Since
\[
\tan^2\left(2^{k - 1}/\gamma_k\right)
\]
converges faster to $1$ than
\[
\tan\left(2^{k - 1}/\gamma_k\right),
\]
we can simplify approximation \eqref{A4RT} as follows:
\begin{equation}
\label{SF4RT} % simplified formula for reciprocal of theta
\frac{1}{\theta_{1,k}} \approx \frac{1}{2}\left(1 - \tan\left(\frac{2^{k - 1}}{\gamma_k}\right)\right).
\end{equation}

Substituting approximation \eqref{SF4RT} into Equation \eqref{R2TTMLF4P}, we can approximate the two-term Machin-like Formula \eqref{MF4P} for $\pi$ as follows:
\[
\frac{\pi}{4} \approx 2^{k - 1}\arctan\left(\frac{1}{\gamma_k}\right) + \arctan\left(\frac{1}{2}\left(1 - \tan \left(\frac{2^{k - 1}}{\gamma_k}\right)\right)\right).
\]
Consequently, from the Lemma \ref{L4RA}, we have the following:
\begin{equation}
\label{MA4PWTF} % modified approximation for pi with tangent function
\frac{\pi}{4} \approx \frac{2^{k - 1}}{\gamma_k} + \frac{1}{2}\left(1 - \tan\left(\frac{2^{k - 1}}{\gamma_k}\right)\right).
\end{equation}

Consider now the following limit: 
\[
\lim_{x\to\pi/4} \frac{{2\sin^2}(x)}{\tan(x)} = 1.
\]
Since the numerator of this limit is very close to its denominator near the vicinity of $\pi/4$, we may replace the tangent function in approximation \eqref{MA4PWTF} with twice the square of the sine function. This leads to
\[
\frac{\pi}{4} \approx \frac{2^{k - 1}}{\gamma_k} + \frac{1}{2}\left(1 - 2\sin^2\left(\frac{2^{k - 1}}{\gamma_k}\right)\right)
\]
or
\[
\frac{\pi}{4} \approx \frac{2^k}{2\gamma_k} + \frac{1}{2}\cos\left(\frac{2^k}{\gamma_k}\right)
\]
or
\begin{equation}
\label{A4PWCF} % approximation for pi with cosine function
\frac{\pi}{2}\approx \frac{2^k}{\gamma_k} + \cos\left(\frac{2^k}{\gamma_k}\right).
\end{equation}

Equation \eqref{A4PWCF} provides a hint for the computation of $\pi$. Using the change in the variable such that
\[
\frac{2^k}{\gamma_k}\to a_1,
\]
we may attempt to compute $\pi$ using the following iteration:
\begin{equation}
\label{IF4PWCC} % iterative formula for pi with cubic convergence
a_n = a_{n - 1} + \cos(a_{n - 1})
\end{equation}
resulting in
\[
\lim_{n\to\infty} a_n = \frac{\pi}{2}.
\]

Although the iteration Formula \eqref{IF4PWCC} is obtained absolutely using a heuristic method, we can prove its cubic convergence to $\pi/2$.

\begin{theorem}
The convergence rate of the iterative Formula \eqref{IF4PWCC} to $\pi/2$ is cubic.
\end{theorem}

\begin{proof}
Let
\[
a_n = \frac{\pi}{2} + \varepsilon_n,
\]
where $\varepsilon_n$ is an error term occurring at step $n$. Then, taking the Maclaurin series expansion of the cosine function, we have
\[
\cos\left(\frac{\pi}{2} + \varepsilon_n\right) = -\varepsilon_n + \frac{\varepsilon_n^3}{6} - \frac{\varepsilon_n^5}{120} + \frac{\varepsilon_n^7}{5040} - \frac{\varepsilon_n^9}{362880} + \ldots. 
\]
Consequently, we get
\[
a_{n + 1} = \frac{\pi}{2} + \varepsilon_n + \cos\left(\frac{\pi}{2} + \varepsilon_n\right) = \frac{\pi}{2} + \varepsilon_n -\varepsilon_n + \frac{\varepsilon_n^3}{6} + O\left(\varepsilon_n^5\right),
\]
leading to cubic convergence, since the next error term becomes
\[
\varepsilon_{n + 1} = \frac{\varepsilon_n^3}{6} + O\left(\varepsilon_n^5\right).
\]
Thus, we can conclude that if $|\varepsilon_{1}| < 1$, then all consecutive error terms 
\[
\varepsilon_{1} > \varepsilon_{2} > \varepsilon_{3} > \ldots \varepsilon_{n} > \varepsilon_{n + 1} > \varepsilon_{n + 2} > \ldots > 0
\]
tend to zero with increasing $n$ at a cubical convergence rate. This completes the proof.
\end{proof}

The following is a Mathematica code for computing digits of $\pi$ with the help of the iterative Formul \eqref{IF4PWCC}.
\vspace{-0.25 cm}
\small
\begin{shaded}
\begin{verbatim}
Clear[a, accNum, n];

(* Initial accuracy number *)
accNum = 10;

(* Initial value of π/2 *)
a = SetAccuracy[3.145926/2, accNum];

Print["--------------------------------------"];
Print["Iteration n", "  |  ", "Number of digits of π"];
Print["--------------------------------------"];

(* Iteration *)
For[n = 1, n <= 8, a = SetAccuracy[a + Cos[a], accNum]; 
    Print[n, "            |  ", MantissaExponent[Pi 
        - 2*a][[2]]//Abs]; accNum = 5*accNum; n++];

Print["--------------------------------------"];
\end{verbatim}
\end{shaded}
\normalsize
\vspace{-0.25 cm}

The output of this code:
\small
\begin{verbatim}
--------------------------------------
Iteration n | Number of digits of π
--------------------------------------
1           | 8
2           | 26
3           | 76
4           | 232
5           | 698
6           | 2095
7           | 6288
8           | 18868
--------------------------------------
\end{verbatim}
\normalsize
shows that the iterative Formula \eqref{IF4PWCC} provides cubic convergence since the number of digits of $\pi$ increases by a factor of $3$ after each iteration.

Another iterative algorithm that triples the number of correct digits of $\pi$ after each cycle of iteration is based on a cubic modular equation \cite{Berggren2004, Borwein1984, Borwein1986}. However, despite extremely rapid convergence, both algorithms are not optimal for computing digits of $\pi$. Specifically, in Equation \eqref{IF4PWCC}, the argument of the cosine function is relatively large, as $a_n$ tends to
$$
\frac{\pi}{2} = 1.5707963267\dots
$$
with increasing $n$. As a result, this requires a very large number of the summation terms in the power expansion of the cosine function. On the other hand, the iterative algorithm based on a cubic modular equation generates irrational (surd) numbers at each iteration cycle. This considerably decelerates computation, since determination of the irrational numbers over and over again at each cycle of iteration is itself a big challenge. Therefore, the algorithms, based on Chudnovsky and Machin-like formulas, were mostly used to beat the records in the history of $\pi$. The detailed chronology of world records in computing digits of $\pi$ can be found in \cite{Agarwal2013}.

\subsection{A Numerical Solution for Nested Radicals with Roots of $2$}

As we can see, approximation \eqref{SDTRA} doubles the number of digits of  $\pi$. However, this approach can be used beyond computing digits of $\pi$. It can also be used to compute nested radicals consisting of square roots of $2$. In particular, using iterative Formula \eqref{NIF}, we can compute nested radicals with roots of $2$ as follows:
\[
\frac{v_k}{v_{k - n}} \approx \frac{\sqrt{2 - c_n}}{c_{n + 1}}.
\]

The Mathematica code below shows how to generate digits of nested radicals with roots of $2$ at $n = 3$:
\[
\frac{v_k}{v_{k - 3}}\approx \frac{\sqrt{2 - c_3}}{c_4} = \frac{\sqrt{2 - \sqrt{2 + \sqrt{2 + \sqrt 2}}}}{\sqrt{2 + \sqrt{2 + \sqrt{2 + \sqrt 2}}}}.
\]
\vspace{-0.25 cm}
\begin{shaded}
\small
\begin{verbatim}
Clear[k, c, flNum, accNum, v]

(*Assign value of k*)
k = 5000;

(* Increase if needed: $RecursionLimit = 100000; *)
$RecursionLimit = 50000;

(* Define nested radicals *)
c[0] := c[0] = 0;
c[n_] := c[n] = SetAccuracy[Sqrt[2 + c[n - 1]], k];

(* Compute floor number *)
flNum = Floor[c[k]/Sqrt[2 - c[k - 1]]];

(*Set accuracy with accuracy number*)
accNum = Length[RealDigits[flNum][[1]]];

(*Compute v_k*)
v[1] := v[1] = SetAccuracy[Floor[flNum], accNum];
v[n_] := v[n] = 1/2*(v[n - 1] - 1/v[n - 1]);

Print[MantissaExponent[v[k]/v[k - 3] - Sqrt[2 - c[3]]/c[4]][[2]] // 
    Abs, " computed digits of nested radical"];
\end{verbatim}
\end{shaded}
\normalsize
\vspace{-0.25 cm}

This code produces the following output:
\small
\begin{verbatim}
1506 computed digits of nested radical
\end{verbatim}
\normalsize

It is interesting to note that due to relation
\[
\frac{v_k}{v_{k - 1}} \approx \frac{\sqrt{2 - \sqrt 2}}{\sqrt{2 + \sqrt 2}} = {\sqrt 2} - 1
\]
we can compute $\sqrt{2}$ and its total numbers of correct digits using the following command lines:
\vspace{-0.25 cm}
\small
\begin{shaded}
\begin{verbatim}
Print["Computed square root of 2 is ", N[1 + v[k]/v[k - 1], 20],"..."];

Print[MantissaExponent[(v[k]/v[k - 1] + 1) - Sqrt[2]][[2]] //
    Abs, " computed digits of square root of 2"];
\end{verbatim}
\end{shaded}
\normalsize
\vspace{-0.25 cm}
This Mathematica command line generates the following output:
\small
\begin{verbatim}
Computed square root of 2 is 1.4142135623730950488...
1506 computed digits of square root of 2
\end{verbatim}
\normalsize

There are several different methods \cite{Nimbran1999, Priestley1999, Kaushik2024} that can be used for computing digits of ${\sqrt 2}$. However, the method of computation that we developed here can be implemented beyond ${\sqrt 2}$ for efficient computation of the nested radicals with roots of $2$ of kind
\[
\frac{\sqrt{2 - c_{n - 1}}}{c_n}.
\]
These nested radicals are utilized in formulas like \eqref{F4PWNR} and \cite{Abrarov2017b}
\[
\frac{\pi}{4} = \sum_{k \ge 2}2^{k - 1}\arctan\left(\frac{\sqrt{2 - c_{k - 1}}}{c_k}\right).
\]
Furthermore, due to relation
\[
\frac{\sqrt{2 - c_{n - 1}}}{c_n} \approx \frac{1}{2}\sqrt{2 - c_{n - 1}}, \qquad n \gg 1,
\]
this technique can also be applied for computation of the nested radicals with roots of $2$ of kind
\[
\sqrt{2 \pm \sqrt{2 + \sqrt{2 + \sqrt{2 + \ldots}}}}.
\]
as an alternative to other known methods \cite{Campbell2023, Cho2025, Servi2003, Zimmerman2008}.

It should be noted that using the methodology described in our publication \cite{Abrarov2017b}, we can also apply nested radicals with roots of $2$ as
\[
\frac{\pi}{4} = 2^{k - 1}\sum_{n \ge k + 1}\arctan\left(\frac{\sqrt{2 - c_{n - 1}}}{c_n}\right)
\]
or
\[
\pi = 2^k\sum_{n \ge k}\arctan\left(\frac{\sqrt{2 - c_{n - 1}}}{c_n}\right),
\]
from which the following limit
\[
\pi = \lim_{k \to \infty} {2^k\sum_{n \ge k}\frac{\sqrt{2 - c_{n - 1}}}{c_n}}
\]
can be obtained.

\subsection{Computation via Arbitrary-Precision Arithmetic}

Consider how Equation \eqref{MTTMLF4P} can be used for computing digits of $\pi$. Rather than to apply it directly for computing digits of $\pi$, we can apply Equation \eqref{ET} to transform the following quotient into reciprocal integers (compare equations \eqref{ITTMLF4P} and \eqref{STMLF4P} above before and after transformation):
\begin{equation}
\label{F4OOT} % formula for one over theta
\frac{1}{\theta_{1,k}} = \frac{v_k - 1}{v_k + 1}
\end{equation}
Another approach to compute $\pi$ is to approximate $\theta_k$ directly during the process of iteration.

The following Mathematica code shows how to compute digits of $\pi$ by approximating Equation \eqref{F4OOT} in the iteration process.

\vspace{-0.25 cm}
\small
\begin{shaded}
\begin{verbatim}
Clear[k, accNum, c, v1, vk, n];

(* Integer k *)
k = 50;

(* Define array of accuracy numbers *)
accNum = {100000, 200000, 300000, 400000, 500000};

(* Define nested radicals *)
c[0] := c[0] = 0;
c[n_] := c[n] = SetAccuracy[Sqrt[2 + c[n - 1]], k];

n = 1;
Do[
  (* Setting accuracy *)
  v1 = SetAccuracy[Floor[c[k]/Sqrt[2 - c[k - 1]]], accNum[[n]]];

  (* Computing v_k *)
  vk = v1; Do[vk = 1/2*(vk - 1/vk), k - 1];

  Print["n = ", n, ", ", MantissaExponent[\[Pi] - 
      4*(2^(k - 1)*ArcTan[1/v1] + ArcTan[(vk - 1)/(vk + 1)])][[2]] //
          Abs, " digits of π"]; n++, 5];
\end{verbatim}
\end{shaded}
\normalsize
\vspace{-0.25 cm}

The output of this code is as follows:
\small
\begin{verbatim}
n = 1, 100014 digits of π
n = 2, 200014 digits of π
n = 3, 300014 digits of π
n = 4, 400014 digits of π
n = 5, 500014 digits of π
\end{verbatim}
\normalsize

The number of digits in the Mathematica code above is determined by a list variable {\bf {\ttfamily accNum}} consisting of a sequence of $5$ numbers, given as follows:
\small
\begin{verbatim}
accNum = {100000, 200000, 300000, 400000, 500000};
\end{verbatim}
\normalsize
If we increase these numbers, the number of correct digits of $\pi$ increases by the same factors. From this example, we can see that by increasing the parameters in the list variable {\bf {\ttfamily accNum}}, we can archive an arbitrary convergence rate.

\section{Determination of the Constant \boldmath{$\theta_{1,k}$}}

Another advantage of the new iterative Formula \eqref{NIF} is its considerably faster computation of the constant $\theta_{1,k}$. Consider as an example $k = 19$. The corresponding two-term Machin like formula for $\pi$ is given as follows:
\begin{equation}
\label{K19}
\begin{aligned}
\frac{\pi}{4} &= 2^{19 - 1}\arctan\left(\frac{1}{\gamma_{19}}\right) - \arctan\left(\frac{1}{\theta_{1,19}}\right) \\
&= 262144\arctan\left(\frac{1}{333772}\right) - \arctan\left(\frac{2519905905\,\ldots\,6547219457}{3573486530\,\ldots\,4866301951}\right),
\end{aligned}
\end{equation}
where the second arctangent term consists of $1\text{,}447\text{,}940$ and $1\text{,}447\text{,}933$ digits in its numerator and denominator, respectively.

The following Mathematica code is built on the two-step iterative Formula
\eqref{TSI}:
\vspace{-0.25 cm}
\small
\begin{shaded}
\begin{verbatim}
Clear[t, k, c, alpha, kappa, lambda, theta, str];

t = AbsoluteTiming[k = 19;

(* Nested radicals with roots of 2 *)
c[0] := c[0] = 0;
c[n_] := c[n] = Sqrt[2 + c[n - 1]];

(* First coefficient *)
alpha = Floor[c[k]/Sqrt[2 - c[k - 1]]];

(* Initial values for two-step iteration *)
kappa = (alpha^2 - 1)/(alpha^2 + 1);
lambda = (2*alpha)/(alpha^2 + 1);

Print["Computing, please wait..."];

(* Two-step iteration *)
Do[x = kappa^2 - lambda^2; y = 2*kappa*lambda;
    kappa = x; Clear[x]; lambda = y; Clear[y], k - 1];

theta = kappa/(1 - lambda);][[1]];

(* Converting to stings *)
strNumer = ToString[Numerator[theta]];
strDenom = ToString[Denominator[theta]];

Print["At k = ", k," the run-time is ", t," seconds"]

(* String for second coefficient *)
str = ToString[Subscript[\[Theta], 1, k], StandardForm];

(* Formatting output *)
If[k <= 5, Print[str," = ", theta], Print[str," = ",
    -StringJoin[StringPart[strNumer, 2 ;; 11], "...",
        StringPart[strNumer,-10 ;; -1]]/StringJoin[StringPart[strDenom,
            2 ;; 11], "..." , StringPart[strDenom, -10 ;; -1]]]];
\end{verbatim}
\end{shaded}
\vspace{-0.25 cm}
\normalsize
\noindent This code returns the following output:
\small
\begin{alltt}
Computing, please wait...
-----------------------------------------
\(\theta\sb{1,19} = -\frac{2519905905\,...\,6547219457}{3573486530\,...\,4866301951}\)
-----------------------------------------
At k = 19 the run-time is 10.3824 s
\end{alltt}
\normalsize
Thus, this code takes about $10$ s to generate the constant $\theta_{1,19}$ for Formula \eqref{K19}.

The next box below shows the Mathematica code implemented on the basis of the new iterative Formula
\eqref{NIF}:
\vspace{-0.25 cm}
\small
\begin{shaded}
\begin{verbatim}
Clear[t, k, c, v, theta, str];

t = AbsoluteTiming[k = 19;

(* Nested radicals with roots of 2 *)
c[0] := c[0] = 0;
c[n_] := c[n] = Sqrt[2 + c[n - 1]];

(* First coefficient *)
v = Floor[c[k]/Sqrt[2 - c[k - 1]]];

Print["Computing, please wait..."];

(* New iterative formula *)
Do[v = 1/2*(v - 1/v), k - 1]; theta = (v + 1)/(v - 1);][[1]];

(* Converting to stings *)
strNumer = ToString[Numerator[theta]];
strDenom = ToString[Denominator[theta]];

Print["At k = ", k," the run-time is ", t," seconds"]

(* String for second coefficient *)
str = ToString[Subscript[\[Theta], 1, k], StandardForm];

(* Formatting output *)
If[k <= 5, Print[str," = ",theta], Print[str," = ",
    -StringJoin[StringPart[strNumer, 2 ;; 11], "...",
        StringPart[strNumer,-10 ;; -1]]/StringJoin[StringPart[strDenom,
            2 ;; 11], "..." , StringPart[strDenom,-10 ;; -1]]]];
\end{verbatim}
\end{shaded}
\vspace{-0.25 cm}
\normalsize
\noindent The output of this code is as follows:
\small
\begin{alltt}
Computing, please wait...
-----------------------------------------
\(\theta\sb{1,19} = -\frac{2519905905\,...\,6547219457}{3573486530\,...\,4866301951}\)
-----------------------------------------
At k = 19 the run-time is 1.98505 s
\end{alltt}
\normalsize
This code generates the same constant $\theta_{1,19}$ after around $2$ s. As we can see, the implementation of the new iterative Formula \eqref{NIF} performs $5$ times faster.

The acceleration by a factor of $5$ remains in a wide range of $k$ values. Specifically, we continued the run-time test up to $k = 27$. To generate the constant $\theta_{1,27}$ (see Equation \eqref{K27}), the code based on the two-step iterative Formula \eqref{TSI} takes almost $4$ h, while the code based on the new iterative Formula \eqref{NIF} requires about $45$ min.

All computations were performed on a typical laptop computer with Intel CORE i3, and 16 GB of RAM. The values of $k$ in the last two Mathematica codes may be modified accordingly from $k = 2$ up to $k = 27$. The value of $k$ can be increased further on laptop/desktop computers with larger RAM capacities.

\section{Conclusions}

An iterative Formula \eqref{NIF} that can be used for computing $\pi$ and nested radicals with roots of $2$ is derived. It is shown how this iterative formula can be implemented to generate and approximate the two-term Machin-like formulas for $\pi$. Some examples with Mathematica code are provided.

\section*{Acknowledgment}

This work was supported by National Research Council Canada, Thoth Technology Inc., York University and Epic College of Technology.

\end{document}